\documentclass[11pt,letterpaper]{amsart}
\usepackage{amsmath}
\usepackage{amssymb}

\usepackage{epsfig}
\usepackage[all, knot]{xy}
\xyoption{arc}

\newtheorem{thm}{Theorem}[section]
\newtheorem*{thmA}{Theorem~A}
\newtheorem*{thmB}{Theorem~B}
\newtheorem*{thmC}{Theorem~C}

\newtheorem{df}[thm]{Definition}
\newtheorem{prop}[thm]{Proposition}
\newtheorem{cor}[thm]{Corollary}

\newtheorem{claim}{Claim}
\newtheorem*{cor*}{Corollary~1.3}

\newtheorem{lem}[thm]{Lemma}
\newtheorem{ex}[thm]{Example}
\newtheorem{rem}[thm]{Remark}
\newcommand{\Pic}{\operatorname{Pic}}
\newcommand{\Spec}{\operatorname{Spec}}

\newcommand{\ord}{\operatorname{ord}}

\newcommand{\Per}{\operatorname{Per}}

\newcommand{\pp}{\mathbb{P}}
\newcommand{\cc}{\mathbb{C}}
\newcommand{\qq}{\mathbb{Q}}
\newcommand{\af}{\mathbb{A}}
\newcommand{\ox}{\mathcal{O}}

\begin{document}

\title[equidistribution of periodic points]
{The equidistribution of small points
for strongly regular pairs of polynomial maps}
\author{Chong Gyu Lee}



\date{\today}

\address{Department of Mathematics, University of Illinois at Chicago, Chicago, IL, 60607, US}

\email{phiel@math.uic.edu}

\maketitle

\begin{abstract}
    In this paper, we prove the equidistribution of periodic points of
    a regular polynomial automorphism $f : \af^n \rightarrow \af^n$ defined over a number field $K$:
    let $f$ be a regular polynomial automorphism defined over a number field $K$ and let $v\in M_K$. Then, there exists an $f$-invariant probability measure $\mu_{f,v}$ on $\operatorname{Berk}\bigl( \pp^n_{\mathbb{C}_v} \bigr)$ such that the set of periodic points of $f$ is equidistributed with respect to $\mu_{f,v}$. We will prove it by equidistribution of small points for strongly regular pair of polynomial maps. 
\end{abstract}

\section{Introduction}

    Let $f =(f_1,\cdots, f_n): \af^n \rightarrow \af^n$ be a polynomial map defined over a number field $K$; i.e, we assume that $f_1,\cdots, f_n \in K[X_1, \cdots, X_n]$. We say \emph{$f$ is a polynomial automorphism} if there is a polynomial map $f^{-1}:\af^n \rightarrow \af^n$ which is the inverse of $f$. We say that a polynomial automorphism $f$ is \emph{regular} if meromorphic extensions $\overline{f}, \overline{f^{-1}}$ of $f, f^{-1}$, which are rational maps on $\pp^n$, have no common indeterminacy point. Regular polynomial automorphisms are well studied in various field \cite{BS1, BS2, BS3, BS4, BS5, BLS, De, M, K, K2, Le1, S1}.

    In this paper, we obtain an arithmetic equidistribution theorem for regular polynomial maps defined over a number field:
    \begin{thmA}
        Let $f:\af^n \rightarrow \af^n$ be a regular polynomial automorphism defined over a number field $K$ and let $v\in M_K$. Then, there exists an $f$-invariant probability measure $\mu_{f,v}$ on the Berkovich projective space $\operatorname{Berk}\bigl( \pp^n_{\mathbb{C}_v} \bigr)$ such that the set of periodic points of $f$ is equidistributed with respect to $\mu_{f,v}$: let $\{ x_m\} $ be a sequence of periodic points of $f$ such that $|\{ x_m\} \cap W|< \infty$ for any proper subvariety $W$ of $\pp^n$. Then
        \[
        \dfrac{1}{\ord x_m} \sum_{y \in \Gamma_{x_m}} \delta_y ~\text{weakly converges to}~\mu_{f,v}
        \]
       where $\Gamma_{x_m}$ is the Galois orbit of $x_m$ and $\delta_y$ is a Dirac measure at $y$.
    \end{thmA}

    Theorem~A generalizes the known result for H\'{e}non maps on $\af^n_\cc$: Bedford, Lyubich \& Smillie \cite{BLS} proved that that the saddle periodic points of H\'{e}non map on $\af^2_\cc$ are equidistributed with respect to the measure of maximal entropy. Dinh and Sibony \cite{DS2} have recently proved the corresponding complex equidistribution result for periodic points of regular polynomial automorphisms on $\af^n_\cc$ for any $n$. In particular, the measure $\mu_{f,v}$ in Theorem~A at the archimedean place coincides with the measure of maximal entropy.

    The measures $\mu_{f,v}$ are Chambert-Loir measures associated to a metrized adelic line bundle \cite{Ch}. To obtain Theorem~A, we use the arithmetic equidistribution results of Yuan \cite{Y} (see Theorem 2.3). As we shall explain, dynamical system of regular polynomial automorphism is not covered by Yuan's application in algebraic dynamics, as Zhang's dynamical adelic metric \cite{Z1} is designed for polarizable endomorphisms on projective varieties while these polynomial automorphisms have indeterminacy in $\pp^n$. In Section~6, we explain how to construct a ``polarizable sequence of morphisms" from a regular polynomial automorphism in order to use the existing arithmetic equidistribution results.

    In fact, Theorem~A is a special case of a more general equidistribution theorem for strongly regular pairs of polynomial maps of the same degree:
    \begin{df}\label{SRO}
        Let $f,g : \af^n \rightarrow \af^n$ be polynomial maps defined over a number field $K$ and let $h$ be the absolute logarithmic height function on $\pp^n_{\overline{K}}$. We say a pair $\{f,g\}$ is \emph{strongly regular} if it satisfies the following condition:
        \begin{itemize}
        \item the meromorphic extensions of $f,g$ have no common indeterminacy point.
        \item $c(f) ,c(g) < \infty$ where
        \[
        c(f): = \displaystyle \limsup_{\substack{h(X) \rightarrow \infty\\X \in \af^n_{\overline{K}} }} \dfrac{h(X)}{h\bigl( f(X) \bigr)}.
        \]
        \item $\deg f, \deg g \geq 2$.
        \item $f\circ g = g \circ f$.
        \item $f,g$ are algebraically stable: $\deg f^m = (\deg f)^m$ and $\deg g^m = (\deg g)^m$.
        \item $\deg (f\circ g) < \min (\deg f, \deg g)$.
        \end{itemize}
    \end{df}
    \noindent Note that we will use the $D$-ratio version of the definition of the strongly regular pair later, which uses $d(f), d(g)$ instead of $c(f), c(g)$ for computational convenience. (See Subsection~2.3 for details.)

    The following theorem shows that a sequence of generic and small points is equidistributed with respect to a canonical measure:
    \begin{thmB}
        Let $S=\{f,g:\af^n \to \af^n\}$ be a strongly regular pair of polynomial maps defined over a number field $K$ such that $\deg f = \deg g$. Then, we have the equidistribution of small points:
        let $\{ x_m\} $ be a generic and small sequence on $\pp^n_{\overline{K}}$ with respect to the arithmetic canonical height function $\widetilde{h}_S$. Then, for any place $v \in M_K$, there is a probability measure $\mu_{S,v}$ on $\operatorname{Berk}\bigl( \pp^n_{\mathbb{C}_v} \bigr)$  such that a sequence of the probability measure on the Galois orbit of $x_m$ weakly converges to $\mu_{S,v}$.
    \end{thmB}

    As in Theorem~A, the measures $\mu_{S,v}$ are Chambert-Loir measures associated to a metrized adelic line bundle $(\ox_{\pp^n}(1), ||\cdot||_S)$. The metric is constructed by forming local Green functions (local escape-rate function) for the pair $S=\{f, g\}$. Note that the local Green function $G_{S,\cdot}$ of a strongly regular pair $S$ is a Weil function so that we have a N\'{e}ron divisor associate to $G_{S,\cdot}$, which is equivalent to a metrized line bundle.

    Now suppose that $f$ is a regular polynomial automorphism on $\af^n$. We apply Theorem~B to obtain Theorem~A, but we observe that the pair $\{f, f^{-1}\}$ is a strongly regular pair of ``different'' degree if $n>2$ - it is easily resolved by Proposition~\ref{degree}. Also, Theorem~B requires a generic and small sequence consisting of periodic points. Theorem~C allows us to apply Theorem~B to obtain Theorem~A. (See Section~7 for details.)

     \begin{thmC}
     Let $f: \af^n \rightarrow \af^n$ be a regular polynomial automorphism defined over a number field $K$. Then, the set of periodic point
     of $f$ is Zariski dense.
     \end{thmC}

    Theorems A and B should be viewed in the context of a collection of ``arithmetic equidistribution" results obtained in recent years; the equidistribution of small points is first studied by Szpiro, Ullmo \& Zhang \cite{SUZ}. After various research of Baker \& Rumely \cite{BR}, Chambert-Loir \cite{Ch} and Favre \& Rivera-Letelier \cite{FR}, Yuan \cite{Y} proved it in arbitrary dimensions. Yuan's result is applied in various cases of algebraic dynamical systems defined over a number field: Yuan applied his theorem for polarizable endomorphisms and the author \cite{Le6} used it for some automorphisms on $K3$ surface.

\par\noindent\emph{Acknowledgements}.\enspace
I would like to thank Shu Kawaguchi for comments on good and bad reductions. I also thank Eric Bedford, Laura DeMarco, and Xinyi Yuan for helpful discussions and comments.

\section{Preliminaries}

    Here's some basic terminology. We refer \cite{MO} for arithmetic height functions on a finitely generated field $K$ over $\qq$, \cite{Le2} for details on the $D$-ratio, \cite{Cu}, \cite{H} and \cite{Sh} for the resolution of indeterminacy, \cite{K2} and \cite{SH} for the local Green functions.

    \subsection{Notations} We set the following notations.\\

    \begin{tabular}{l@{\quad}l}
        $I(f)$ &  the indeterminacy locus of a rational map $f:\pp^n \dashrightarrow \pp^n$\\
        $K$ &  a number field\\
        $v$ & a prime place on $K$\\
        $M_K$  & the set of prime places on $K$\\
        $M_K^0$ & the set of nonarchimedean places on $K$\\
        $M_K^\infty$ & the set of archimedean places on $K$\\
        $K_v$ & the $v$-adic local field\\
        $\ox_v = \ox_{K_v}$ & the ring of integer of $K_v$\\
        $||\cdot||_v$ & the $v$-adic supreme norm on $\af^{n+1}_{\mathbb{C}_v}$\\
        $||\cdot||_{L} = \{ ||\cdot||_{L, v} ~|~ v\in M_K \}$ & an adelic metric $\left\{ ||\cdot||_{L, v} ~|~ v \in M_K \right\}$ on a line bundle $L$\\
        $\mathbf{k}=\mathbf{k}_v$ & the residue field of $K_v$\\
        $\widetilde{f}$ & reduction map of $f$ on $\pp^n_{\mathbf{k}_v}$\\
    \end{tabular}

    \subsection{Equidistribution of small points}

    One of main ingredients of this paper is the equidistribution of small points. In this section, we introduce the theorem of Yuan \cite{Y}. Note that it requires the semipositive metric.

    \begin{thm}[Yuan, \cite{Y}]\label{ED}
    Let $V$ be a projective variety defined over a number field $K$, let $\overline{L}= (L,||\cdot||)$ be a metrized line bundle such that $L$ is an ample line bundle on $V$ and $||\cdot||$ is a semipositive adelic metric on $L$ and let $v\in M_K$. Define an arithmetic height function $h_{\overline{L}}$ on the set of closed subvarieties $W$ of $V$,
    \[
    h_{\overline{L}}(W) = \dfrac{\widehat{c}_1 (\overline{L}|_W)^{\dim W+1} }{(\dim W +1 )\ord_L W}
    \]
    and a probability measure on the Berkovich analytic space  $\operatorname{Berk}\bigl( V(\cc_v) \bigr)$,
    \[
    \mu_{\overline{L}, v} = \dfrac{c_1(\overline{L})^{\deg V}_v }{\deg_{\overline{L}}V},
    \]
    where $\widehat{c}_1 (\overline{L})$ is the arithmetic first Chern form of $\overline{L}$. (See \cite[\S2.1]{BGS} for details.) Suppose that $\{x_m\}$ is a generic and small sequence. (Note that a sequence $\{x_m\}$ is called \emph{generic} if $\{x_m\}$ has finite intersection with any proper subvariety of $V$ and is called \emph{small} if $h_{\overline{L}}(x_m)$ converges to $h_{\overline{L}}(V)$.)
    Then, the sequence of probability measures on the Galois orbit of $x_m$ weakly converges to $\mu_{\overline{L},v}$ on $\operatorname{Berk}\bigl( \pp^n_{\cc_v} \bigr)$.
    \end{thm}

    \subsection{The $D$-ratio of polynomial maps}

    In this subsection, we introduce the basic idea of the $D$-ratio. We will let $K$ be a number field and let $h$ be the absolute naive height height on $\pp^n_{\overline{K}}$ for convenience.

    In Definition~\ref{SRO}, we want that $c(f), c(g)$ are finite. The reason why we require this condition is that we will use the following global and local height inequalities:
    \begin{lem}
    Let $f: \af^n \rightarrow \af^n$ be a polynomial map defined over a number field $K$ and let $v\in M_K$. Define
    \[
    c(f): = \displaystyle \limsup_{\substack{h(X) \rightarrow \infty \\X \in \af^n_{\overline{K}}}} \dfrac{h(X)}{h\bigl( f(X) \bigr)}.
    \]
    Suppose that $c(f)$ is finite. Then, for any $\epsilon >0$, there is a constant $C_\epsilon$ such that
    \[
    \left( c(f) + \epsilon \right) h\bigl( f(X) \bigr)  >  h(X) - C_\epsilon
    \quad \text{and}\quad
    \left( c(f) + \epsilon \right) \log^+ ||f(X)||_v  >  \log^+ ||X||_v - C_\epsilon
    \]
    for all $X \in \af^n_{\overline{K}}$.
    \end{lem}
    \begin{proof}
    The first inequality is trivial because of the definition of $c(f)$. For the second inequality, consider the resolution of indeterminacy of $f$:
    \[
    \xymatrix{
    V \ar[d]^\pi \ar[dr]^\phi & \\
    \pp^n \ar@{-->}[r]_f & \pp^n
    }
    \]
    Let $E = (c(f) + \frac{\epsilon}{2})\phi^*H - \pi^*H$ and let $D$ be an ample divisor on $V$. Then the fractional limit of $E$, defined on \cite{Le4}, is positive: let $X = \pi(Y)$. Then we get the following inequality.
    \begin{eqnarray*}
    \operatorname{Flim}_D\left(E,\pi^{-1}\bigl( \af^n \bigr) \right)
    &=&
    \limsup_{\substack{h_D(Y) \rightarrow \infty \\ Y \in \pi^{-1}\bigl( \af^n \bigr)}} \dfrac{h_{E}(Y)}{h_D(Y)} \\
    & = &
    \limsup_{\substack{h_D(Y) \rightarrow \infty \\ X = \pi(Y) \in \af^n_{\overline{K}}}} \dfrac{(c(f)+\frac{\epsilon}{2})h_{\phi^*H}(X) - h_{\pi^*H}(X)}{h_D(X)} \\
     &\geq& 0.
    \end{eqnarray*}
    Therefore, $E$ is a pseudo-effective divisor whose base locus lies outside of $\pi^{-1}\bigl( \af^n \bigr)$. Thus, a sum of a pseudo-effective divisor and a big divisor,
    \[
    E' = (c(f) + \epsilon)\phi^*H - \pi^*H = E + \frac{\epsilon}{2}\phi^*H
    \]
    is an effective divisor whose base locus lies outside of $\pi^{-1} \bigl( \af^n \bigr)\setminus |\phi^*H| = \pi^{-1}\bigl( \af^n \bigr)$. Thus, $\lambda_{E',v}$ is bounded below on $\pi^{-1}\bigl( \af^n \bigr)$. Let $X = \pi(Y)$ and get
    \begin{eqnarray*}
     \lambda_{E',v}(Y)  &=& (c(f) + \epsilon)\lambda_{H,v} \bigl( \phi(Y)\bigr) - \lambda_{H,v} \bigl( \pi(Y)\bigr) \\
    &=& (c(f) + \epsilon)\lambda_{H,v} \bigl( f(X)\bigr) - \lambda_{H,v} \bigl( X \bigr)  \\
    &>& O(1)
    \end{eqnarray*}
    by the functorial property of arithmetic height functions. Therefore, we get the desired result because $\lambda_{H,v}(X) = \log^+ ||X||_v +O(1)$.
    \end{proof}

    It is not convenient to consider $\epsilon$ in every inequality. So, we will use the $D$-ratio version of the definition of strongly regular
     pairs (Definition~\ref{SR}). It is a stronger condition than the original definition because we have the following inequality:
     \[
     c(f) \leq \dfrac{r(f)}{\deg f}.
      \]
     (see \cite[Theorem~7.3]{Le2} for details.) Also, it is more convenient for regular polynomial automorphisms: if we only consider the case of a regular polynomial automorphism, then polynomial maps $\Phi, \Psi$ defined in the proof of Theorem~B are regular polynomial automorphisms again and hence we know the $D$-ratios of $\Phi$ and $\Psi$. In general, it is very painful to show that the $D$-ratios of $\Phi, \Psi$ are finite while it is clear that $c(\Phi), c(\Psi)$ are finite if $r(f), r(g)$ are finite. It is the only reason why we define the strongly regular pair with $c(f)$.

     Note that we consider a polynomial map $f :\af^n \rightarrow \af^n$ as a rational map $f:\pp^n \dashrightarrow \pp^n$ such that $I(f)$ is contained in the infinity hyperplane $H = \pp^n \setminus \af^n$. We can define the $D$-ratio for such rational maps.

    \begin{df}\label{d-ratio}
        Let $H$ be a hyperplane on $\pp^n$ and let $f:\pp^n \dashrightarrow \pp^n$ be a rational map, defined over a number field $K$, such that $I(f) \subset H$. Then, we can define the $D$-ratio of $f$:
        let $V$ be a successive blowup of $\pp^n$ with a birational morphism $\pi:V \rightarrow \pp^n$ such that $f \circ \pi$ extends to a morphism $\phi$ and let
        \[\{H^\#, E_1, \cdots, E_r \}
        \]
         be a basis of $\Pic(V)$ where
        $H^\#$ is the proper transformation of $H$ by $\pi$ and $E_1, \cdots, E_r$ are irreducible components of the exceptional divisor on $V$. We define \emph{the $D$-ratio of $f$ on $V$} to be
        \[
        r(f) := \deg f \times \max{\dfrac{a_i}{b_i}}
        \]
        where $\pi^*H = a_0 H^\# + \sum a_i E_i$ and $\phi^*H = b_0 H^\# + \sum b_i E_i$.
    \end{df}

     \begin{df}[The $D$-ratio version]\label{SR}
        We say a pair of polynomial maps $\{f,g : \af^n \rightarrow \af^n\}$, defined over a number field $K$, is \emph{strongly regular} if it satisfies the following conditions:
        \begin{itemize}
        \item the meromorphic extensions of $f,g$ have no common indeterminacy point.
        \item $r(f) ,r(g) < \infty$.
        \item $\deg f, \deg g \geq 2$.
        \item $f\circ g = g \circ f$.
        \item $f,g$ are algebraically stable: $\deg f^m = (\deg f)^m, \deg g^m = (\deg g)^m$.
        \item $\deg (f\circ g) < \min (\deg f, \deg g)$.
        \end{itemize}
    \end{df}

    \begin{rem}
        If you want to use Definition~\ref{SRO} instead of Definition~\ref{SR}, all $d(f),d(g)$ will be replaced by $\deg f \cdot \bigl( c(f) + \epsilon\bigr)$ and $\deg g \cdot \bigl( c(g) + \epsilon \bigr)$. For example, the $v$-adic norm inequality {\rm (\ref{rineq})} in Section~3
        \[
        C_1||X||^{\frac{d_f}{r(f)} }  \leq ||f(X) || \leq C_2 ||X||^{d_f}
        \]
        will be replaced by
        \[
        C_{1,\epsilon} ||X||^{\frac{1}{c(f)+\epsilon} }  \leq ||f(X) || \leq C_2 ||X||^{d_f}.
        \]
    \end{rem}

    \begin{lem}\label{Le1}
        Let $H$ be a hyperplane on $\pp^n$, let $f:\pp^n \dashrightarrow \pp^n$ be a rational map, defined over a number field $K$, such that
        $I(f) \subset H$, let $\af^n_{\overline{K}} = \pp^n_{\overline{K}} \setminus H$, let $r(f)$ be a $D$-ratio of $f$ and let $v\in M_K$.
        Then, there are constant $C$ and $C_v$ satisfying
        \[\dfrac{r(f)}{\deg f}h\bigl( f(X) \bigr)  >  h(X) - C
        \quad \text{for all }X \in \af^n_{\overline{K}}, \quad
        \dfrac{r(f)}{\deg f} \log^+ || f(X)||_v  >  \log^+ ||X||_v - C_v
         \]
        for all $X \in \af^n_{\overline{K}}$.
    \end{lem}
    \begin{proof}
    For the first inequality, see \cite[Theorem 5.1]{Le2}. For the second one, let $\pi:V \to \pp^n$ be a birational morphism such that $\phi = f \circ \pi$ is a morphism providing a $D$-ratio $r(f)$ on $V$. Then,
    \[
    E' = \dfrac{r(f)}{\deg f}\phi^*H - \pi^*H
    \]
    is an effective divisor whose base locus lies outside of $\pi^{-1}(\af^n)$. Therefore, $\lambda_{E', v}$ is bounded below on $\pi^{-1}(\af^n)$. Thus, we have
    \[
        \dfrac{r(f)}{\deg f} \lambda_{\phi^*H ,v}(Y)  >  \lambda_{\pi^*H ,v}(Y) - C \quad \text{for all }Y \in \pi^{-1}\bigl(\af^n_{\overline{K}}\bigr).
    \]
    Let $X = \pi(Y)$. Since $\lambda_{H,v} = ||X||_v +O(1)$ and $\phi(Y) = f(X)$, we get the desired result.
    \end{proof}

\section{Nonarchimedean Green functions for a strongly regular pair}

    Kawaguchi \cite{K2} showed that the local Green functions of a regular polynomial automorphism have nice properties. In Section~3,4 and 5, we show that the local Green functions of a strongly regular pair also have similar properties.
    In this section, we study the local Green functions of a strongly regular pair of polynomial maps at a nonarchimedean place $v$. Note that we use the topology induced by a metric $|\cdot| = |\cdot|_v$ on $K_v$, not Zariski topology.
    We refer \cite{K2} for details on the local Green functions of polynomial maps.

    \begin{prop}
    Let $f:\af^n \rightarrow \af^n$ be a polynomial map defined over $K$. Then, \emph{the local Green function of $f$ at $v$} is well-defined:
    \[
    G_f (X) = G_{f,v}(X) := \lim_{m \rightarrow \infty} \dfrac{1}{\deg f^m} \log^+ ||f^m(X)||.
    \]
    Moreover, we have
    \[
    G_f(X) \leq \log^+ ||X|| + C
    \]
    for some constant $C$.
    \end{prop}
    \begin{proof}
    See \cite[\S.1]{K2}.
    \end{proof}

    \begin{thm}\label{two}
        Let $K_v$ be an algebraically closed field with nontrivial nonarchimedean
        valuation and let $S = \{ f, g :\af^n \rightarrow \af^n\}$ be a strongly regular pair of polynomial maps defined over $K_v$. Then there are
        open subsets $V_f, V_g$ of $\af^n_{K_v}$ and constants $C_f, C_g$ with the following properties.
        \begin{enumerate}
        \item  $V_f \cup V_g  = \af^n_{K_v}$.
        \item  $G_f (X) \geq \log^+||X|| + C_f$ on $V_f.$
        \item  $G_g (X)  \geq \log^+||X||  + C_g$ on $V_g.$
        \end{enumerate}
    \end{thm}
    \begin{proof}
        \begin{enumerate}
        \item  To find such open sets, we need clarify the relation between $f,g$ and $f\circ g$.  Let $d_f, d_g$ be degrees of $f$ and $g$ respectively. Since $f,g$ are polynomial maps, we can consider $f,g$ as rational maps on $\pp^n$ and write
        \[
        f(X) = \left[ X_0^{d_f}, F_1, \cdots, F_n \right], \quad g(X) = \left[ X_0^{d_g}, G_1, \cdots, G_n \right]
        \]
        where $F_i, G_i$ are homogeneous polynomials of degree $d_f$ and $d_g$ respectively. Because $f,g$ are jointly regular, $\{X_0, F_i,G_i\}_{i\in I }$ has no nontrivial common zeros. Hence, by Hilbert's Nullstellensatz, we have an integer $M$ and homogeneous polynomials $P_{ij},Q_{ij}, R_{j}$ satisfying
        \begin{equation}\tag{A}\label{HN}
        \sum_{i=1}^n P_{ij}(X) F_i(X) + \sum_{i=1}^n Q_{ij}(X) G_i(X) + X_0 R_j(X) = X_j^M.
        \end{equation}
        Note that $\deg P_{ij} = M - d_f, \deg Q_{ij} = M - d_g$ and $\deg R_j = M-1$.

        By the assumption $\deg {f\circ g} < \min (d_f , d_g)$ and $f\circ g = g \circ f$, all component of $f\circ g$ and $g\circ f$ has a common divisor of degree $l = d_f d_g - \deg (f\circ g)$. Moreover, since the first component is $X_0^{d_fd_g}$, the common divisor is $X_0^l$:
        \begin{eqnarray*}
        f \circ g  &=&  \left[ X_0^{d_f d_g},F_1\bigl( g(X) \bigr), \cdots, F_n\left( g(X) \right) \right] \\
        &=& \left[ X_0^{d_f d_g-l},F_1\bigl( g(X) \bigr)/X_0^l, \cdots, F_n\bigl( g(X) \bigr) /X_0^l \right]
        \end{eqnarray*}
        and
        \begin{eqnarray*}
        g \circ f &=& \left[ X_0^{d_f d_g},G_1 \bigl( f(X) \bigr), \cdots, G_n(f(X)) \right] \\
        &=& \left[ X_0^{d_f d_g-l},G_1\bigl( f(X) \bigr)/X_0^l, \cdots, G_n\bigl( f(X) \bigr) /X_0^l \right].
        \end{eqnarray*}
        Hence, we get
        \begin{equation}\label{comp}\tag{B} F_i\bigl( g(X) \bigr) = G_i\bigl( f(X) \bigr) = X_0^l J_i(X)\quad \text{for all}~i=0,\cdots, n
        \end{equation} where $J_i(X)$ are homogeneous polynomials of degree $d_J = d_fd_g - l = \deg (f\circ g)$.

        Let
        \[
        C_0: = \max \bigl\{ |\alpha| ~|~ \alpha=1~\text{or a coefficients of one of }F_i, G_i, P_{ij}, Q_{ij},R_j, J_i\bigr\}
        \]
        and take two positive numbers
        \[
        \epsilon = C_0^{-4}, \quad\text{and}\quad \delta = C_0^{-2}.
        \]
        Then, $\epsilon, \delta$ are positive constants satisfying the following:
        \begin{equation}\tag{C}\label{epde}
        \begin{cases}
        \epsilon < \dfrac{1}{C_0}, \delta < \dfrac{1}{C_0 }, \\
        \epsilon^{d - d_J } C_0 < \delta ~\text{where}~ d=\min (d_f, d_g)~\text{and}\\
        \epsilon^l C_0 < \delta^2 ~\text{where}~ l={d_f d_g - d_J}.
        \end{cases}
        \end{equation}
        Note that $d_f d_g\geq 2$ and $d_J < \min (d_f, d_g)$ so that $l$ is an integer larger than $2$.

        Define two open sets
        \[
        V_f : =  V_{f,\epsilon, \delta} = \left\{ X \in \af^n_{K_v} ~\left|~ ||X||< \dfrac{1}{\epsilon} ~\text{or}~ \max\bigl(||f(X)||,1\bigr) > \delta \max\bigl(||X||^{d_f}, 1\bigr) \right. \right\}
        \]
        and
        \[
        V_g := V_{g,\epsilon, \delta} = \left\{ X  \in \af^n_{K_v} ~\left|~ ||X||< \dfrac{1}{\epsilon} ~\text{or}~ \max\bigl(||g(X)||,1\bigr) > \delta \max\bigl(||X||^{d_g}, 1\bigr) \right. \right\}.
        \]

        Then, $V_f \cup V_g  = \af^n_{K_v}$: suppose there is a point $X=[1,X_1, \cdots, X_n] \in \af^n \setminus (V_f \cup V_g)$. Then $X$ satisfies the following three inequalities
        \[
        ||X|| \geq \dfrac{1}{\epsilon} >1, \quad \delta \max\bigl(||X||^{d_f}, 1\bigr) \geq \max\bigl(||f(X)||,1 \bigr)
         \]
         and
         \[
         \delta \max\bigl(||X||^{d_g}, 1\bigr) \geq \max\bigl(||g(X)||,1\bigr).
        \]
        We may assume that $||X|| = \sup |X_i| = |X_n|$. Applying $||\cdot||$ to $(\ref{HN})$ when $j=n$, we get
        \begin{eqnarray*}
        ||X||^M= |X_n^M|  &=& \left\| \sum_i P_{in} f_i + \sum_i Q_{in} g_i + R_n \right\| \\
        &\leq& \max \bigl( ||P_{in}||\cdot ||f(X)||, ||Q_{in}||\cdot  ||g(X)||, |R_n|  \bigr)\\
        &\leq& \max \bigl(  C_0  ||X||^{M-d_f} \cdot \delta ||X||^{d_f}, C_0 ||X||^{M-d_g}\cdot \delta ||X||^{d_g}, C_0 ||X||^{M-1} \bigr)\\
        &<& ||X||^M
        \end{eqnarray*}
        which is a contradiction.\\

        \item[(2),(3)] We start with the following claims.
        \begin{claim}\label{inclusion}
        \[f(V_{f}) \subset V_{f}.\]
        \end{claim}
        \begin{proof}[Proof of Claim~\ref{inclusion}]
         Let $X = [1, X_1, \cdots, X_n]\in V_f$ be a point such that $f(X) \not \in V_f$. By definition of $V_f$, $f(X)$ should satisfy two inequalities:
        \begin{equation}\label{f(X)}\tag{D}
        ||f(X)|| \geq \dfrac{1}{\epsilon} \quad \text{and} \quad  \max\bigl(||f\bigl( f(X) \bigr)||,1\bigr) \leq \delta \max\bigl(||f(X)||^{d_f}, 1) .
        \end{equation}

        We showed $\af^n = V_f \cup V_g$ on (1). So, $f(X)$, which lies outside of $V_f$, is an element of $V_g$: $X$ satisfies either
         \[
         ||f(X)||< \dfrac{1}{\epsilon} \quad \text{or} \quad \delta \max\bigl(||f(X)||^{d_g}, 1\bigr) < \max\bigl(||g\bigl(f(X) \bigr)||,1\bigr).
         \]
        However, because of (\ref{f(X)}), the first one cannot happen: $X$ should satisfy $(\ref{f(X)})$ and
        \begin{equation}\label{Vg}\tag{E} \delta \max\bigl(||f(X)||^{d_g}, 1\bigr) < \max\bigl(||g(f(X))||,1\bigr).
        \end{equation}

        Let's show that there is no such $X \in V_f$ satisfying (\ref{f(X)}) and (\ref{Vg}). $X \in V_f$ satisfies either
        \[
        ||X|| \geq \dfrac{1}{\epsilon} \quad \text{or} \quad \max\bigl(||f(X)||,1) > \delta \max\bigl(||X||^{d_f}, 1).
        \]
        If $||X|| < \dfrac{1}{\epsilon}$, then consider $||g\bigl(f(X)\bigr) ||$. Since $f\circ g$ is a polynomial map of degree $d_J$, we have
        \[
        ||g\bigl(f(X)\bigr) || = \max ||1^l J_i(X)|| \leq C_0 ||X||^{d_J} \leq C_0 \left( \dfrac{1}{\epsilon} \right)^{d_J}.
        \]
        On the other hand, because of (\ref{Vg}) and the first inequality in (\ref{f(X)}), we have
        \[
        \max\bigl(||g\bigl(f(X) \bigr)||,1\bigr) > \delta \max\bigl(||f(X)||^{d_g}, 1\bigr) > \delta ||f(X)||^{d_g} > \delta  \left( \dfrac{1}{\epsilon} \right)^{d_g}.
        \]
        So, we get $C_0 \epsilon^{d_g - d_J} > \delta $.
        But, from a property of $\epsilon$ and $\delta$, we have
        \[
        \epsilon^{d_g - d_J} C_0 \leq \epsilon^{d-d_J} C_0< \delta \quad\text{where }d = \min(d_f,d_g),
        \]
        which is a contradiction. So, $X$ should satisfy
        \[
        ||X|| \geq \dfrac{1}{\epsilon} \quad \text{and} \quad \max\bigl(||f(X)||,1) > \delta \max\bigl(||X||^{d_f}, 1).
         \]

        However, combine it with $(\ref{Vg})$ and get
        \begin{equation}\label{AA}\tag{F} \max\bigl(||g(f(X))||,1\bigr) \geq \delta \max\bigl(||f(X)||^{d_g}, 1\bigr) \geq \delta^2 \max\bigl(||X||^{d_fd_g} ,1).
        \end{equation}
        On the other hand, we have
        \begin{equation}\label{BB}\tag{G}
        \max\bigl(||g(f(X))||,1\bigr) \leq C_0 ||X||^{d_J} \leq C_0 \epsilon^l ||X||^{d_f d_g  }.
        \end{equation}
        Thus, (\ref{AA}) and (\ref{BB}) induce that $C_0 \epsilon^l > \delta^2$, which contradicts to (B).
     \end{proof}

    \begin{claim}\label{const1}
        For any $X \in V_f$, we have
        \[
        \max\bigl(||f(X)||,1\bigr) \geq C_f\bigl(\max||X||^{d_f} , 1\bigr)
        \]
        where $C_f = \min \{\delta, \epsilon^{d_f}\}.$
    \end{claim}

    \begin{proof}
        Let $X\in V_{f}$: $X$ satisfies either
        \[
        ||X|| \leq \dfrac{1}{\epsilon}  \quad \text{or} \quad \max (||f(X)||,1) \geq \delta \max ( ||X||^{d_f}, 1).
        \]
        If the former one holds, then we have
        \[
        C_f ||X||^{d_f} \leq C_f \left( \dfrac{1}{\epsilon} \right)^{d_f} \leq 1 \leq \max(||f(X)||,1).
        \]
        For the latter case, it is clear because $C_f \leq \delta$.
    \end{proof}

    \vspace{5mm}

    Let's complete the proof of $(2)$ and $(3)$ by using Claim~1 and Claim~2. Let $X\in V_f$. Then, by Claim~\ref{inclusion},
    $f^m(X) \in V_{f}$ for all $m\geq 0$. Apply the logarithm to the inequality in Claim~\ref{const1} and get
        \[
        \log^+||f(Y_m)|| \geq d_f \log^+||Y_m|| + \log C_f
        \]
        for all $Y_m = f^m(X)\in V_{f}$. Therefore, we get the desired result by the telescoping sum:
        \begin{eqnarray*}
        G_f(X) &=& \lim_{m \rightarrow \infty} \dfrac{1}{d_f^m} \log^+ ||f^m(X)|| \\
        &=&  \log^+||X|| + \sum_{m=1}^\infty \left( \dfrac{1}{d_f^m} \log^+ ||f^m(X)|| - \dfrac{1}{d_f^{m-1}} \log^+ ||f^{m-1}(X)|| \right) \\
        &\geq& \log^+||X|| + \sum_{m=1}^\infty \dfrac{1}{d_f^{m-1}} C_f.
        \end{eqnarray*}
        \end{enumerate}
    \end{proof}

    Recall that the local Green functions are defined by the pointwise limit. It is not enough for the construction of a semipositive metric later - we need that the dynamical metric is semipositive to get the equidistribution. So, we want that a sequence $\left\{\dfrac{1}{d_f^m} \log^+ ||f^m(X)||\right\}$ is uniform convergent.

    \begin{prop}\label{uniform}
        \[
        G_f= \displaystyle \lim_{m \rightarrow \infty} \dfrac{1}{d_f^m} \log^+ ||f^m(X)||
         \quad \text{and} \quad
         G_g= \displaystyle \lim_{m \rightarrow \infty} \dfrac{1}{d_g^m} \log^+ ||g^m(X)||
         \]
         are uniformly convergent on $V_f$ and $V_g$ respectively.
    \end{prop}
    \begin{proof}
        Let $X \in V_f$: $X$ satisfies either
        \[
        ||X|| < \epsilon^{-1} \quad\text{or} \quad \max (||f(X)||,1) > \delta \max(||X||^{d_f},1).
        \]

        By Lemma~\ref{Le1}, we have
        \begin{equation}\label{rineq}\tag{H}
        C_1 ||X||^{\frac{d_f}{r(f)}} \leq  ||f(X)|| \leq C_2 ||X||^{d_f}
        \end{equation}
        where $r(f)$ is the $D$-ratio of $f$, which is not smaller than 1 \cite[Proposition~4.5]{Le2}. Divide (\ref{rineq}) by $||X||^{d_f}$ to get
        \[
        C_1 ||X||^{\frac{d_f}{r(f)} - d_f} \leq  \dfrac{||f(X)||}{||X||^{d_f}} \leq C_2.
        \]

        If $||X|| < C_3 = \epsilon^{-1}$, then the lower bound is smaller than $C_1 C_3^{\frac{d_f}{r(f)} - d_f}$ because $\frac{d_f}{r(f)} - d_f \leq 0$ and $||X|| < C_3$. Take the logarithm and the absolute value to get
        \[
        \left| \log \dfrac{||f(X)||}{||X||^{d_f}}  \right| \leq  \max\left( \bigl| \log  C_1   \bigr| + \left( {d_f} - \dfrac{{d_f}}{r(f)} \right) \bigl| \log  C_3   \bigr|, \bigl| \log  C_2   \bigr| \right).
        \]

        If $||X|| \geq \epsilon^{-1}$ and $\max (||f(X)||,1) > \delta \max(||X||^{d_f},1)$, then $\delta||X||^{d_f} \geq \delta \epsilon^{-d_f} \geq 1$. Thus, we have
        \[
        \max (||f(X)||,1) > \delta ||X||^{d_f} \geq 1,
        \]
        which guarantees $||f(X)||>1$. Hence, we get
        \[
        ||f(X)|| > \delta||X||^{d_f} \quad \Rightarrow \quad
        \delta \leq  \dfrac{||f(X)||}{||X||^{d_f}} \leq C_2
        \]
        and
        \[
        \left| \log \dfrac{||f(X)||}{||X||^{d_f}}  \right| \leq  \max\left(  -\log \delta, \bigl| \log C_2 \bigr|, \right).
        \]

        Let
        \[
        M = \max\left( \bigl| \log  C_1   \bigr| + \left( {d_f} - \dfrac{{d_f}}{r(\phi)} \right) \bigl| \log  C_3   \bigr|, \bigl| \log  C_2   \bigr|, -\log \delta\right).
        \]
        Then, by the telescoping sum, we get
        \begin{eqnarray*}
            \left| \dfrac{1}{{d_f}^{m}} \log ||f^{m}(X)|| - \dfrac{1}{{d_f}^{m+l}} \log ||f^{m+l}(X)||  \right|
            &\leq & \sum_{j=m}^{m+l}  \dfrac{1}{{d_f}^j}  \left| \log \dfrac{||f^{j+1}(X)||}{||f^j(X)||^{d_f}}  \right|   \\
            &\leq & \sum_{j=m}^{m+l}  \dfrac{1}{{d_f}^j} M
        \end{eqnarray*}
        and hence $\dfrac{1}{{d_f}^m}\log||f^m(X)||$ is uniformly convergent on $V_f$. Similarly,
        $\dfrac{1}{{d_g}^m}\log||g^m(X)||$ is uniformly convergent on $V_g$.
    \end{proof}

\section{Good reduction}

    To check the convergence of a sequence of adelic metrics $\langle \{||\cdot||_{L,v,m}~|~v\in M_K\} \rangle_{m\geq 0}$, we need that $\langle ||\cdot||_{L,v,m}\rangle_{m\geq 0}$ is trivial for all but finitely many $v\in M_K$. In this section, we check the concept ``good reduction'', which guarantees that local Green functions exactly equal to the local height function at all but finitely may prime spots. In Section~6,  we will use this fact to show that the dynamical adelic metric id well defined. We refer \cite{K2, S2} for details of good reduction.

    \begin{df}\label{good reduction}
        Let $S = \{ f, g :\af^n \rightarrow \af^n\}$ be a strongly regular pair of polynomial maps defined over a number field $K$. and let $v\in M_K^0$. We say that
        \emph{$S$ have good reduction at $v$} if
        \begin{enumerate}
            \item Every coefficient of $f,g$ is a $v$-adic integer so that we can extend $f,g$ to polynomial maps defined over $\ox_v$:
            \[
            f_v, g_v : \af^n_{\ox_v} \rightarrow \af^n_{\ox_v}.
            \]
            \item $\deg f_v = \deg f,  \deg g_v = \deg g.$
            \item The pair of reduction maps $\{\widetilde{f},\widetilde{g}\}$ is jointly regular on $\pp^n_{\overline{\mathbf{k}}_v}$, i.e., $\widetilde{f},\widetilde{g}$ have no common indeterminacy point on $\pp^n_{\overline{\mathbf{k}}_v}$.
        \end{enumerate}
    \end{df}

    It is quite easy to find prime spots $v$ which satisfy (1) and (2): if we find $v$ such that every coefficient of $f$ and $g$ is a $v$-adic unit, then $v$ satisfies (1) and (2). For the condition (3), we will find the equivalent condition using the Elimination Theory.

    \begin{lem}[Elimination Theory]\label{elimination}
        Let $\mathbf{k}$ be a field, let $\overline{\mathbf{k}}$ be an algebraic closure of $\mathbf{k}$, let $A$ be a commutative ring with unity, let $f_1, \cdots, f_t$ be homogeneous polynomials in $A[X_0, \cdots , X_n]$ and let $\rho : A \rightarrow \mathbf{k}$ be a homomorphism. We extend $\rho$ to $A[X_0, \cdots X_n]$ by applying $\rho$ to the
        coefficients of a polynomial. Then the followings are equivalent:
            \begin{enumerate}
                \item
                    \[ \left\{
                        \rho\left(r\right)  ~|~ r\in A, r(X_0, X_1, \cdots, X_n)^l \subset (f_1, \cdots f_t)~\text{for some}~l\geq 1
                        \right\}  = \{ 0\}
                    \]
                \item There is a $P\in \pp^n_{\overline{\mathbf{k}}_v}$ such that $f_i(X)=0$ for all $i=1,\cdots, t$.
            \end{enumerate}
    \end{lem}
    \begin{proof}
        See \cite[I.5.7.A]{H} or \cite[Vol II, \S 80]{V}.
    \end{proof}

    \begin{prop}\label{3}
        Let $S = \{ f, g :\af^n \rightarrow \af^n\}$ be a strongly regular pair of polynomial maps defined over a number field $K$. Suppose that $S$ satisfies Definition~\ref{good reduction}(1), (2)  Then, the followings are equivalent.
        \begin{enumerate}
            \item $S$ has good reduction at $v \in M_K$, i.e., $S$ satisfies Definition~\ref{good reduction}(3).
            \item As ideals in $\ox_v[X_0, \cdots, X_n]$, we have
            \[
            (X_0, X_1, \cdots X_n)^N \subset \bigl( X_0, F_1(X), \cdots, F_n(X), G_1(X), \cdots, G_n(X) \bigr)
            \]
            for some integer $N$.
        \end{enumerate}
    \end{prop}
    \begin{proof}
        Take a set
        \[
        I:=\left\{
        r\in \ox_v ~|~ r(X_0, X_1, \cdots, X_n)^l \subset (X_0, F_1(X),\cdots F_n(X), G_1(X), \cdots, G_n(X))~\text{for some}~l\geq 1
        \right\}.
        \]
        Then, Lemma~\ref{elimination} says that $\rho(I)=0$ if and only if $\widetilde{X}_0, \widetilde{F}_1, \cdots, \widetilde{F}_n, \widetilde{G}_1, \cdots, \widetilde{G}_n$ have a nontrivial common root in $\pp^n_{\overline{K}}$.
        Thus, $\{\widetilde{f},\widetilde{g}\}$ is jointly regular if and only if there exists $r \in I \cap \ox_{k_v}^*$.
    \end{proof}

    We will finish this section by showing that a strongly regular pair $S=\{f,g\}$ has good reduction at all but finitely many $v \in M_K$ and the local Green functions are trivial if $S$ has good reduction at $v \in M_K$.

    \begin{prop}
        Let $S = \{ f, g :\af^n \rightarrow \af^n\}$ be a strongly regular pair of polynomial maps defined over a number field $K$. Then, there is a finite subset $\mathcal{B}$ of $M_K$ such that $S$ has good reduction for all $v \not \in \mathcal{B}$.
    \end{prop}
    \begin{proof}
        Let $\mathcal{B}$ be a set of prime places $v$ such that $v$ is archimedean or $v$ satisfies the following conditions:
        \begin{itemize}
        \item all coefficients of $F_i, G_i, P_{ij}, Q_{ij}, R_j, J_i$ are $\mathcal{B}$-integers in $\ox_K$.
        \item Let $\rho_v : (\ox_K)_{\mathcal{B}} \rightarrow \mathbf{k}_v$ be the natural map. Then,
        \[
        \deg f = \deg (\rho_v(f)) \quad \text{and} \quad \deg g = \deg (\rho_v(g)).
        \]
        \end{itemize}
        Then, for any $v \not \in \mathcal{B}$, $S_v= \{ f_v, g_v : \af^n_{K_v} \rightarrow \af^n_{K_v}\}$ satisfies Definition~\ref{good reduction} (1), (2). Also, the second condition and the existence of $P_{ij}, Q_{ij}, R_j$ guarantee that $S_v$ satisfies Proposition~\ref{3} (2) so that $S$ has good reduction at $v$.
    \end{proof}

    \begin{prop}\label{good}
        Suppose that a strongly regular pair $S = \{ f, g :\af^n \rightarrow \af^n\}$ of polynomial maps defined over a number field $K$ has good reduction at $v\in M^0_K$. Then,
        \begin{enumerate}
            \item $G_{f, v} \leq \log ^+ ||\cdot ||_v$
            \item $G_{g, v} \leq \log ^+ ||\cdot ||_v$
            \item $G_{f, v} \geq \log ^+ ||\cdot ||_v$ on $V_{f,v,1,1}$
            \item $G_{f, v} \geq \log ^+ ||\cdot ||_v$ on $V_{g,v,1,1}$
        \end{enumerate}
    \end{prop}
    \begin{proof}
        Suppose that $S = \{ f, g :\af^n \rightarrow \af^n\}$ has good reduction at $v$. Then, by Proposition~\ref{3}, it is clear that all $F_i, G_i, P_{ij}, Q_{ij}$, $R_j, J_i$ have coefficients in $\ox_v^*$, all coefficients in equalities in Section~3 will be of the norm $1$ and hence all inequality will be proved.
    \end{proof}

\section{The local Green functions at archimedean places}

    We only have finitely many archimedean places and hence the local Green functions need not be trivial. Still, we need the uniform convergence. Fortunately, we can prove the archimedean version of Theorem~\ref{two} by the triangle inequality so that we can prove the desired result.

    \begin{thm}\label{arch}
        Let $K_\infty$ be an algebraically closed field with archimedean norm $||\cdot||_{\infty}$ and let $S = \{ f, g :\af^n \rightarrow \af^n\}$ be a strongly regular pair. Then there are two open subsets $V_f, V_g$ of $\af^n$ and constants $C_f, C_g$ with the following properties.
        \begin{enumerate}
        \item  $G_f, G_g (X) \leq \log^+||X||_{\infty} + C_\infty.$
        \item  $f(V_f) \subset V_f, g(V_g) \subset V_g$ and $V_f \cup V_g  = \af^n$.
        \item  $G_f (X) \geq \log^+||X||_{\infty} + C_f$ on $V_f$ and $G_g (X)  \geq \log^+||X||_{\infty}  + C_g$ on $V_g.$
        \end{enumerate}
    \end{thm}
    \begin{proof}
        Sibony \cite{Si} proved that a rational map has the local Green function at archimedean place. So, we don't worry about the convergence.
        \begin{enumerate}
            \item Let $||f||_{\infty}$ be the maximum value of the coefficients of $f$. Then,
            \[
            \log^+||f(X)||_{\infty} \leq \log M_q||f||_{\infty} ||X||^q_{\infty}
            \]
            where $M$ is thee number of degree $q$ monomials of $(n+1)$-variables.
            \item It is totally same with the proof of Theorem~\ref{two} with new constant $C_\infty$.
            \item Let $P_{ij}$, $Q_{ij}$ and $R_{ij}$ are polynomials defined on previous section. Then, Theorem~\ref{two}~(2,3) exactly works with new constant
            \[
            C_\infty = (2n+1)M_m \max\left\{ ||P||_{\infty}, ||Q||_{\infty}, ||R||_{\infty}, 1 \right\}
            \]
            where $M_m$ is the number of degree $m$ monomials of degree $m$ and of $(n+1)$-variables.
        \end{enumerate}
    \end{proof}

    \begin{prop}\label{uniform2}
        $G_f$ is uniformly convergent on $V_f$ and $G_g$ is uniformly convergent on $V_g$
    \end{prop}
    \begin{proof}
        By Theorem~\ref{arch}, we can mimic the proof of Proposition~\ref{uniform}.
    \end{proof}

\section{Equidistribution for strongly regular pairs}\label{equidistribution}

    In this section, we prove the equidistribution of small points for a strongly regular pair $\{f,g\}$ such that $\deg f = \deg g$. The main idea is to build a sequence of morphisms made by product of iterations of $f$ and $g$. Silverman \cite{S4} suggested the way of defining a morphisms from a jointly regular family of polynomial maps. In particular, if we have a jointly regular pair $\{f,g\}$ with $\deg f = \deg g$, we can easily make a morphism.
    \begin{lem}\label{joining}
        Let $S = \{ f, g :\af^n \rightarrow \af^n\}$ be a jointly regular pair of polynomial maps. Suppose $\deg f = \deg g$. Then, a polynomial map
        \[
        \phi = \phi_S = (f, g) : \af^n \rightarrow \af^n \times \af^n, \quad X \mapsto \bigl( f(X), g(X)\bigr)
        \]
        extends to a morphisms $\widetilde{\phi}_S : \pp^n \rightarrow \pp^{2n}.$
    \end{lem}
    \begin{proof}
        Let $X=(X_0, \cdots, X_n)$. $f$ and $g$ are polynomial maps so that we can write
        \[
        f = [X_0^{d}, F_1(X), \cdots, F_n(X)], \quad g = [X_0^{d}, G_1(X), \cdots, G_n(X)]
        \]
        where $F_i, G_i$ are homogeneous polynomials of degree $d = \deg f = \deg g$.

        Define a polynomial map
        \[
        \phi = (F_1, \cdots, F_n, G_1, \cdots G_n) : \af^n \rightarrow \af^{2n}.
        \]
        Then, since $f,g$ do not share the same indeterminacy points, there is no common zero of $\{F_i, G_i\}_{i=1, \cdots, n}$ except trivial one. Therefore, $\phi$ will extends to a morphism:
        \[
        \widetilde{\phi}(X) := [X_0^d, F_1(X), \cdots, F_n(X), G_1(X), \cdots, G_n(X)].
        \]
    \end{proof}

     The following Lemmas show that Lemma~\ref{joining} also works for iterations of $f$ and $g$: we can define $\phi_{\{f^m, g^m\}}$, too.

    \begin{lem}\label{fmgm}
     Let $\zeta_1, \zeta_2:\af^n \rightarrow \af^n$ be polynomial maps and let $H = \pp^n \setminus \af^n$ be the infinity hyperplane. Then,
    \[
    \deg (\zeta_1\circ \zeta_2) \leq \deg \zeta_1\cdot \deg \zeta_2 \quad \text{if and only if} \quad
    \zeta_1 (H \setminus I(\zeta_1)) \subset I(\zeta_2).
    \]
    \end{lem}
    \begin{proof}
    See \cite[Lemma~7.8]{S2}.
    \end{proof}

    \begin{lem}\label{joint}
    Let $S = \{ f, g :\af^n \rightarrow \af^n\}$ be a strongly regular pair of polynomial maps defined over a number field $K$ and let $I(f), I(g)$ be the indeterminacy locus of $f$ and $g$ respectively. Then, $I(f^m) \subset I(f)$ and $I(g^m) \subset I(g)$.
    \end{lem}
    \begin{proof}
    Since $S$ is strongly regular, we have
    \[
    \deg (f \circ g ) < \min (\deg f, \deg g) , \quad \deg f, \deg g \geq 2 \quad \text{and} \quad I(f) \cap I(g) = \emptyset.
    \]
    Thus, we get
    \[
    \deg (f \circ g ) < \min (\deg f, \deg g) < \deg f \cdot \deg g
    \]
    and hence
    \[
    f (H \setminus I(f)) \subset I(g)
    \]
    by Lemma~\ref{fmgm}. Moreover,
    \[
    f (H \setminus I(f)) \in H \setminus I(f)
    \]
    because $I(g) \subset H$ and $I(f)\cap I(g) = \emptyset$.

    Thus, by induction, $f^m(X) = f \bigl( f^{m-1}(X) \bigr)$ is well defined and $f^m(X) \not \in I(f)$ if $X \in H \setminus I(f)$. Hence, we get
    \[
    H \setminus I(f) \subset H \setminus I(f^{m}).
    \]
    \end{proof}
    \begin{cor}\label{again}
    Let $S = \{ f, g :\af^n \rightarrow \af^n\}$ be a strongly regular pair of polynomial maps defined over a number field $K$. Then, $\{f^m,g^m\}$ is also strongly regular for all $m\geq 2$.
    \end{cor}
    \begin{proof}
    The only difficult part of the proof is to show whether $\{f^m, g^m\}$ is still jointly regular. By Lemma~\ref{joint}, we have $I(f^m) \subset I(f), I(g^m) \subset I(g)$ and hence
    \[
    I(f^m) \cap I(g^m) \subset I(f)\cap I(g) = \emptyset.
    \]
    Rest of condition to be strongly regular are easy to check:
    \begin{itemize}
    \item $f^m, g^m$ are still algebraically stable and commutative. Also, $\deg f^m = (\deg f)^m \geq 2$,
    \item We have
    \[
    \dfrac{r(f\circ g)}{\deg (f \circ g)} \leq \dfrac{r(f)}{\deg f}\dfrac{r(g)}{\deg g}.
    \]
    Moreover, $f,g$ are algebraically stable. So we get
    \[
    r(f^m) \leq \deg f^m  \cdot \left( \frac{r(f)}{\deg f} \right)^m = r(f)^m \leq \infty.
    \]
    \item The degree condition is easily gained because because $f\circ g = g \circ f$:
    \begin{eqnarray*}
    \deg \left( f^m \circ g^m \right) &=& \deg (f \circ g)^m \\
    &\leq& \left( \deg f \circ g \right)^m \\
    &<&  \min \left( \deg f , \deg g \right)^m \\
    &=&  \min \left( \deg f^m,\deg g^m \right) .
    \end{eqnarray*}
    \end{itemize}
    \end{proof}

    Because of Lemma~\ref{joining} and Corollary~\ref{again}, we can define a sequence of morphisms:
    \begin{thm}\label{main}
        Let $S = \{ f, g :\af^n \rightarrow \af^n\}$ be a strongly regular pair of polynomial maps defined over a number field $K$ such that $\deg f = \deg g$. Define a polynomial map
        \[
        \phi_m : \af^n \rightarrow \af^{2n} \quad X \mapsto (f^m(X), g^m(X)) \quad \text{for all}~m\in \mathbb{N}.
        \]
        Consider $\phi_m$ as a morphisms from $\pp^n$ to $\pp^{2n}$. Then, the sequence $\{\phi_m\}$ generates the dynamical adelic metric for $S$ on $\ox_{\pp^n}(1)$.
    \end{thm}
    \begin{proof}

        Let $L= \ox_{\pp^{n}}(1)$ and $L' = \ox_{\pp^{2n}}(1)$ be ample line bundles on $\pp^n$ and $\pp^{2n}$ respectively,
        let $||\cdot||_v$ be the $v$-adic supreme norm on $\af^N_K$ and
        let $\left\{ ||\cdot||_{L, v} \right\}$, $\left\{ ||\cdot||_{L', v} \right\}$ be the standard adelic metrics on $L$ and $L'$ respectively.

        Let's construct a sequence of adelic metrics: because of Lemma~\ref{joining},
        \[
        \phi_m : \pp^n \rightarrow \pp^{2n}
        \]
        are morphisms of degree $d^m = \deg f^m = \deg g^m$. Since $\phi_m^* L'$ is isomorphic to $L^{\otimes d^m}$, we can choose an isomorphism
        \[
        \sigma_m : L^{\otimes d^m} \rightarrow \phi_m^* L'
        \]
        and get a sequence of adelic metrics on $L$
        \[
        ||\cdot||_{L, m,v} =  \sigma_m^* \phi_m^* ||\cdot||_{L',v}^\frac{1}{d^m}.
        \]
        More precisely, consider $\phi_m$ as a polynomial map between line bundles on $\pp^n$ and $\pp^{2n}$,
        \[
        \phi_m : L \approx \af^{n+1} \rightarrow L' \approx \af^{2n+1}
        \]
        induced by the original morphism $\phi_m : \pp^n \rightarrow \pp^{2n}$. Also, let $X = (X_1, \cdots, X_n) = [1,X_0, \cdots, X_n] \in \af^n_{\overline{K}}$ and let $Y = (1,X_0,\cdots X_n) \in \af^{n+1}_{\overline{K}}$.

        Choose a section $s : \pp^n \rightarrow \af^{n+1}$ and get
        \[
        ||s||_{L, m,v}(X) = \sigma_m^* \phi_m^* ||s||_{L',v}^\frac{1}{d^m}(X) =  ||\phi_m^*s||_{L,v}^\frac{1}{d^m}(X)
        = ||\phi_m\bigl( s(X) \bigr)||_{v}^\frac{1}{d^m} .
        \]

        We say that an adelic metric $\{||\cdot||_{L, m}\}$ converges to $\{||\cdot||_{L}\}$ if there is an open set $U$ of $\Spec \ox_K$ such that
        $||\cdot||_{L,m,v}  = ||\cdot||_{L,v}$ for all $v \in U$ and
        $\dfrac{||\cdot||_{L,m,v} }{||\cdot||_{L,v} }$ uniformly converges to $1$ for all $ v \not \in U$.
        Let $U= \mathcal{B}^c$ where $\mathcal{B}$ is the finite set of primes found in Proposition~4.4. Then, $S$ has good reduction at all $v \in U$ and hence we get
        \[
        \dfrac{1}{d} \log^+||f(X)||_v = \dfrac{1}{d} \log^+||g(X)||_v = \log^+||X||_v
        = \log ||Y||_v \quad \text{for all}~X\in \af^n_{\overline{K}}
        \]
        by Proposition~\ref{good}. Moreover, since $||\cdot||_v$ is the supreme norm and $\phi_m(X) = (f^m(X), g^m(X))$, we get
        \[
        ||\phi_m(Y)||_v = ||\bigl( 1, f^m(X), g^m(X) \bigr)||_v = \max \left[ \max \left(  1, ||f^m(X)||_v\right),  \max \left(1, ||g^m(X)||_v \right) \right].
        \]
        and hence
        \[
        ||\phi_m(Y)||_v = ||Y||_v^{d^m}.
        \]

        Since $X = [1, X_1, \cdots, X_n]$, we get
        \[
        s(X) = \alpha Y = (\alpha(X), \alpha(X) X_1, \cdots, \alpha(X) X_n)
        \]
        for some $\alpha\in \overline{K}^*$. Moreover, because $||\phi_m(Y)||_v = ||Y||_v^{d^m}$, we get
        \[
        ||\phi_m\bigl( s(X) \bigr))||_v = ||\alpha||_v^{d^m} \cdot ||\phi_m(Y)||_v
        = ||\alpha||_v^{d^m} \cdot ||Y||_v^{d^m}= ||s(X)||_v^{d^m}.
        \]
        and hence
        \[
        \sigma_m^* \phi_m^*||s||_{L,v}^\frac{1}{d^m}(X) = ||\phi_m\bigl( s(X) \bigr)||_{v}^\frac{1}{d^m} =||s(X)||_{v} =||s||_{L',v}(X)
        \]
        for all $m\geq 1$.

        Now, check the uniform convergence of the given metric at $v \not \in U$. Consider again $\phi_m : \pp^n \rightarrow \pp^{2n}$ as a polynomial map
        \[
        \phi_m : \af^{n+1} \rightarrow \af^{2n+1}.
        \]
        Let
        \[
        \pi : \af^{n+1} \setminus \{0\} \rightarrow \pp^{n}.
        \]
        Clearly, $\phi_m$ sends the origin to the origin. Moreover, we have the following equality:
        \[
        \log||\phi_m(Y_1)||_v - d^m \log||Y_1||_v = \log||\phi_m(Y_2)||_v - d^m \log||Y_2||_v
        \]
        if $\pi(Y_1) = \pi(Y_2)$ because $Y_1 = \alpha \cdot Y_2$ for some $\alpha \in K^*$ and $\phi_m$ consists of homogeneous polynomials of degree $d^m$. Therefore,
        the value of $\log||\phi_m\bigl( s(X) \bigr)||_v - d^m \log||s(X)||_v$ only depends on $X$, independent of choice of sections $s : \pp^n \rightarrow {L}$. Thus, choose an arbitrary section $s:\pp^n \rightarrow {L}$ and define continuous functions
        \[
        \tau_{m,v}(X) = \left[  \log \sigma_m^*\phi_m^*|| s||^{\frac{1}{d^m}}_{L', m, v }  - \log ||s||_{L, v} \right](X).
        \]
        Now, we only have to show that $\tau_{m,v}$ is uniformly convergent on $\pp^n$ for all $v \not \in U$: if it is true, then we have the uniform limit of semipositive adelic metrics
        \[
        ||\cdot ||_{L, \infty, v} = \left( \lim_{m \rightarrow \infty} \exp(\tau_{m,v})  \right) \cdot ||\cdot||_{L,v},
        \]
        which is also semipositive.

        \begin{claim}\label{pointconv}
        $\{\tau_{k,v}\}$ is pointwisely convergent on $\pp^{n}_{\overline{K}}$.
        \end{claim}
        \begin{proof}[Proof of Claim~\ref{pointconv}]
        If $X \in \af^n_{\overline{K}}$, then, we easily get the limit:
        \[
        \lim_{m\to \infty} \tau_{m,v}(X) = \exp \left[ \max\left(G_{f,v}(X), G_{g,v}(X) \right) / \log^+||X||_v \right].
        \]

        Suppose that $X \in H = \pp^n_{\overline{K}} \setminus \af^n_{\overline{K}}$. Then, we have three cases:
        \[
        X \in I(f), \quad X \in I(g) \quad \text{or}\quad X \not \in I(f)\cup I(g).
        \]
        Recall Lemma~\ref{fmgm}: since $f,g$ are algebraically stable, we get
        \[
        f\bigl(H \setminus I(f) \bigr) \subset I(g) \quad \text{and} \quad g\bigl(H \setminus I(g) \bigr) \subset I(f).
        \]
        In particular, $f$ sends $I(g)$ to $I(g)$ so that we can consider $f$ as an endomorphism on $I(g)$.

        \noindent{\bf Case 1} : $X \in I(f)$ or $X \in I(g)$.

        If $X = [0,X_1, \cdots, X_n] \in I(f)$, then $F_i(X)=0$ for all $i=1, \cdots, n$. Moreover, since $I(f) \cap I(g) = \emptyset$ and $g\bigl( H \setminus I(g) \bigr) \subset I(f)$, we get $g^m(X) \in I(f)$ for all $m\geq 1$.

        Let $Y = (0, X_0, \cdots, X_n)$. Since $f(Y) = (0,\cdots, 0)$, we get
        \[
        \phi_m(Y) = (0, f^m(X), g^m(X)) = (0,0,\cdots, 0, g^m(X))
        \]
        so that
        \[
        \log ||\phi_m(Y)||_v = \log ||g^m(X)||_v.
        \]

        Lemma~\ref{joint} says that $\deg (f\circ g ) < \max (\deg f m \deg g)$ guarantees that
        \[
        g\bigl( I(f) \bigr) \subset g\bigl( H \setminus H \bigr) \subset I(f).
        \]
        Since $I(f)$ is a union of subvarieties, we have the following height equality \cite[Theorem~B.2.7]{SH}: there is a constant $C_1$ such that
        \begin{equation}\tag{I}\label{Ig}
        \left| h\bigl( g(X) \bigr) - \deg g \cdot h(X)\right| < C_1, \quad \left| \log ||g(X)||_v - \deg g \cdot \log ||X||_v \right| < C_1
        ~\text{for all}~X\in I(f).
        \end{equation}

        Therefore, $\{\tau_{m,v}\}$ is a Cauchy sequence: by telescoping sum, we get
        \begin{eqnarray*}
        &&\left| \left( \dfrac{1}{d^m} \log ||\phi_m(Y)||_v - \log||Y||_v \right) -\left( \dfrac{1}{d^M} \log ||\phi_M(Y)||_v - \log||Y||_v \right) \right| \\
        &&~\leq \sum_{j=m}^{M-1}
        \left| \left( \dfrac{1}{d^{j+1}} \log ||\phi_{j+1}(Y)||_v - \log||Y||_v \right) -\left( \dfrac{1}{d^j} \log ||\phi_j(Y)||_v - \log||Y||_v \right) \right| \\
        &&~\leq \sum_{j=m}^{M-1} \dfrac{1}{d^{j+1}} \left|  \log ||g\bigl(g^j(X)\bigr)||_v - \deg g \cdot \log ||g^j(X)||_v  \right| \\
        &&~\leq \sum_{j=m}^{M-1} \dfrac{1}{d^{j+1}} C_1.
        \end{eqnarray*}

        Similarly, we have the same result if $X \in I(g)$: there is a constant $C_2$ such that
        \begin{equation}\tag{J}\label{If}
        \left| h\bigl( f(X) \bigr) - \deg f \cdot h(X)\right| < C_1, \quad \left| \log ||f(X)||_v - \deg f \cdot \log ||X||_v \right| < C_2
        \end{equation}
        and
        \[
        \left| \left( \dfrac{1}{d^m} \log ||\phi_m(Y)||_v - \log||Y||_v \right) -\left( \dfrac{1}{d^M} \log ||\phi_M(Y)||_v - \log||Y||_v \right) \right|
        \leq \sum_{j=m}^{M-1} \dfrac{1}{d^{j+1}} C_2.
        \]

        \noindent{\bf Case 2} : $X \not \in I(f) \cup  I(g)$.

        Then, $f(X) \in I(g)$ and $g(X) \in I(f)$. Then, $f(X) \in I(g)$ and $g(X) \in I(f)$. So, both (\ref{Ig}) and (\ref{If}) hold. Moreover, we have
        \[
        \log ||\phi_m(Y)||_v = \log ||(0,f^m(X), g^m(X))||_v = \max \left( \log ||f^m(X)||_v, \log ||g^m(X)||_v\right).
        \]

        Therefore, we get
        \[
        \left| \left( \dfrac{1}{d^m} \log ||\phi_m(Y)||_v - \log||Y||_v \right) -\left( \dfrac{1}{d^M} \log ||\phi_M(Y)||_v - \log||Y||_v \right) \right|
        \leq \sum_{j=m}^{M-1} \dfrac{1}{d^{j+1}} \max( C_1, C_2).
        \]
        \end{proof}

        By Claim~\ref{pointconv}, $\tau_{m,v}$ is pointwisely convergent. Thus, if it is uniformly convergent on a dense subset of $\pp^n$, then it is uniformly convergent on $\pp^n$. (See \cite{Kr}.)  Thus, we only have to show that it is uniformly convergent on $\af^n$.

        Define new polynomial maps
        \[
        \Phi := (f, g) : \af^n \times \af^n \rightarrow \af^n \times \af^n
        \]
        \[
        \Psi := (g, f) : \af^n \times \af^n \rightarrow \af^n \times \af^n
        \]
        \[
        \iota: \af^n \rightarrow \af^{2n}, \quad X \mapsto (X,X)
        \]
        Then, $\{ \Phi, \Psi \}$ is a strongly regular by Definition~\ref{SRO}:
            \begin{itemize}
                \item  $I(\Phi) \cap I(\Psi) = \{[0, X_1, \cdots, X_n, X'_1, \cdot, X'_n] ~|~ [0, X_1, \cdots, X_n],[0, X'_1, \cdots, X'_n]\in I(f)\cap I(g) \} = \emptyset$
                \item  Let $Z = (X,X') \in \af^{2n}$. By Theorem~\ref{Le1}, we have
                \[
                 h\bigl( \Phi(Z) \bigr) \geq h(f(X)) + h(g(X')) \geq  \dfrac{d}{r(f)} h(X) + \dfrac{d}{r(f)} h(X') \geq \dfrac{d}{2\max(r(f),r(g))} h(Z)
                 \]
                 for all $Z \in \af^{2n}$. Thus,
                $c(\Phi), c(\Psi) \leq \frac{2}{d} \max \bigl( r(f), r(g) \bigr) < \infty$
                \item  $\deg \Phi = \deg \Psi = \deg f >2$
                \item  $\Phi \circ \Psi = (f \circ g , g \circ f) = (g \circ f , f \circ g) = \Psi \circ \Phi$
                \item  $\Phi, \Psi$ are algebraically stable.
                \item  $\deg \Phi = \deg f > \deg (f \circ g) = \deg (\Phi \circ \Psi)$
            \end{itemize}
        Remark that if $f,g$ are iterations of a regular polynomial automorphism and its inverse, then they are regular polynomial automorphisms again and hence they have finite $D$-ratios so that we can use Definition~\ref{SR}.

        Therefore, for any $v \in M_K$, there exist two positive integers $\epsilon, \delta$ such that
        \begin{equation}\label{union}\tag{K}
        V_{\Phi_,v,\epsilon,\delta} \cup V_{\Psi_,v,\epsilon,\delta} = \af^{2n}
        \end{equation}
        where
        \begin{eqnarray*}
        V_\Phi = V_{\Phi,v,\epsilon,\delta} &:=& \left\{Z=(X,X') \in \af^n \times \af^n ~|~  ||Z||_v < \dfrac{1}{\epsilon} ~\text{or}~
        \max\{ ||\Phi(Z)||_v,1\} \geq \delta \max \{||Z||_v^d, 1\} \right\}
        \end{eqnarray*}
        and
        \begin{eqnarray*}
        V_\Psi = V_{\Psi,v,\epsilon,\delta} &=& \left\{Z=(X,X') \in \af^n \times \af^n ~|~  ||Z|| > \dfrac{1}{\epsilon}  ~\text{or}~
        \max\{ ||\Psi(Z)||_v,1\} \geq \delta \max \{||Z||_v^d, 1\} \right\}.
        \end{eqnarray*}
        By definitions of $\Phi$ and $\iota$, we get
        \[
        ||\Phi(Z)||_v = \max \{||f(X)||_v, ||g(X')||_v\}\quad\text{where}~ ||Z||_v = \max\{ ||X||_v , ||X'||_v\}
        \]
        and
        \[
        ||\Phi\bigl( \iota (X) \bigr)||_v = \max\left( ||f(X)||_v, ||g(X)||_v\right) = ||\Psi\bigl( \iota(X) \bigr) ||_v.
        \]
        By (\ref{union}), $\iota(X) = (X,X) \in V_{\Phi_,v,\epsilon,\delta}$ or $(X,X) \in V_{\Psi_,v,\epsilon,\delta}$. Moreover, by definition of $\Phi$ and $\Psi$, we have
        \[
        ||\Phi\bigl(\iota(X) \bigr) = \max\bigl( ||f(X)||_v , ||g(X)||_v\bigr)= \max\bigl( ||g(X)||_v , ||f(X)||_v\bigr)=||\Psi\bigl(\iota(X) \bigr)||_v.
        \]
        Thus, if $\iota(X) \in V_{\Psi_,v,\epsilon,\delta}$, then $\iota (X) \in V_{\Phi_,v,\epsilon,\delta}$ too by definition of $V_\Phi$ and $V_\Psi$. Furthermore, $V_F \cup V_g = \af^n$ so that $\iota(X)$ should be contained in one of them. Hence,
        \[
        (X,X) \in V_{\Phi} \cap V_{\Psi}.
        \]
        So, $\dfrac{1}{d^m} \log||\Phi^m(\iota(X))|| - \log ||\iota(X)||$ is uniformly convergent on $\iota\bigl( \af^n \bigr)$. Furthermore, because
        \[
        \phi_m = \Phi^m \circ \iota,
        \]
        we have the following: a sequences of continuous functions
        \begin{eqnarray*}
        \tau_{m,v}(X) &=&  \left( \dfrac{1}{d^m} \log \sigma_m^* \phi_m^* || \cdot ||_{L,v} - \log ||\cdot||_{L,v}  \right)(X) \\
        &=& G_{\Phi,v} (X,X)  - \log ||\cdot||_{L',v} (X,X)
        \end{eqnarray*}
        is uniformly convergent to a continuous function $\tau_v$ on $\af^n$.

        Therefore, we have the semipositive metric
        \[
        ||\cdot||_{L,\infty, v} =
        \begin{cases}
        ||\cdot||_{L} & \text{if}~ v  \in U \\
        ||\cdot||_{L}  \cdot e^{\tau_v} & \text{otherwise}
        \end{cases}
        \]
        on $L = \ox_{\pp^n}(1)$.
    \end{proof}

    Now, the existence of the dynamical adelic metric allows the equidistribution.

    \begin{thmB}
        Let $S=\{f,g:\af^n \to \af^n\}$ be a strongly regular pair of polynomial maps defined over a number field $K$ such that $\deg f = \deg g$. Then, we have the equidistribution of small points:
        let $\{ x_m\} $ be a generic and small sequence on $\pp^n_{\overline{K}}$ with respect to the arithmetic canonical height function $\widetilde{h}_S$. Then, for any place $v \in M_K$, there is a probability measure $\mu_{S,v}$ on $\operatorname{Berk}\bigl( \pp^n_{\mathbb{C}_v} \bigr)$  such that a sequence of the probability measure on the Galois orbit of $x_m$ weakly converges to $\mu_{S,v}$.
    \end{thmB}
    \begin{proof}
        Theorem~\ref{main} says that we have that the semipositive dynamical adelic metric. Therefore, by Yuan's result (Theorem~\ref{ED}), we have the equidistribution
        of small points with respect to $\overline{\ox_{\pp^n}(1)} = (\ox_{\pp^n}(1), ||\cdot||_{\ox_{\pp^n}(1),\infty})$.
    \end{proof}

\section{Application to regular polynomial automorphisms}

    In the previous section, we have the equidistribution of small points for the dynamical system of a strongly regular pair of affine morphisms having the same degree. In this section, we will apply this result to regular polynomial automorphisms. As stated in introduction, $\{f,f^{-1}\}$ is an example of a strongly regular pair, but not of the same degree:

    \begin{ex}\label{l>1}
    Let $f$ be the following regular polynomial automorphism of degree $8$
    \[
    f(x_1,x_2, x_3, y_1, y_2, y_3, y_4) = (x_3,x_1+x_3^2, x_2+x_1^2, y_2 - y_1^2, y_3 - (y_2- y_1^2)^2, y_2 - (y_3 - (y_2- y_1^2)^2)^2, y_1).
    \]
    Then, its inverse is of degree $4$:
    \[
    f^{-1}(x,y,z) = (y-x^2, z- (y-x^2)^2, x, y_4, y_1+y_4^2, y_2 + y_1^2, y_3 + y_2^2).
    \]
    \end{ex}

    The following proposition says that we can get a strongly regular pair of the same degree by choosing appropriate iterates of $f$ and $f^{-1}$:

    \begin{prop}\label{degree}
        Let $f:\af^n \rightarrow \af^n$ be a regular affine automorphism with the inverse $f^{-1}$ and let $I(f), I(f^{-1})$ be indeterminacy loci of the meromorphic extensions of $f$ and $f^{-1}$ on $\pp^n_K$ respectively. Let
         $l_1 = \dim I(f)+ 1$ and $l_2 = \dim I(f^{-1}) +1$. Then,
        \[
        l_1 + l_2 = n \quad \text{and} \quad \deg f^{l_2} = \deg f^{-l_1}.
        \]
    \end{prop}
    \begin{proof}
        See \cite[Proposition~2.3.2]{Si} or \cite[Theorem~7.10]{S2}.
    \end{proof}

    \begin{lem}\label{pick}
        Let $f:\af^n \rightarrow \af^n$ be a regular polynomial automorphism with the inverse $f^{-1}$ and let $l_1, l_2$ be positive integers defined in Proposition~\ref{degree}. Then, $\{f^{l_2}, f^{-l_1}\}$ is a strongly regular pair.
    \end{lem}
    \begin{proof}
        It is clear that $f,f^{-1}$ is algebraically stable and $r(f^n) = \deg f^n \cdot deg f^{-n}$ Also, by Lemma~\ref{joint}, we have $I(f^{l_2}) \subset I(f)$ and $I(f^{-l_1}) \subset I(f^{-1})$. So, it is enough to check the degree condition.

        Because Proposition~\ref{degree}, $\deg f^{l_1} = \deg f^{-l_2}$. Thus,
        \[
        \deg (f^{l_2} \circ f^{-l_1}) = \deg (f^{l_2 - l_1}) < \min \left( \deg f^{-l_1}, \deg f^{l_2} \right).
        \]
    \end{proof}

    So $\{f^{l_2} , f^{-l_1}\}$ has the dynamical equidistribution of small points. Especially, we want to show the equidistribution of periodic points, which means we want to find a generic and small sequence of periodic points. It will be done by proving that $\Per(f)$ is Zariski dense. We will modify Fakhruddin's proof for polarizable morphism case \cite{Fa}, using Fornaess \& Sibony's result in complex dynamics \cite{FS}.

\begin{thm}\label{finite}
    Let $\mathcal{K}$ be an algebraically closed field and let $\psi_1, \psi_2 : \pp^n_{\mathcal{K}} \rightarrow \pp^m_{\mathcal{K}}$ be morphisms. Suppose that $\deg \psi_1 \neq \deg \psi_2$. Then
    \[
    \{ X \in \pp^n_{\mathcal{K}} ~|~ \psi_1(X) = \psi_2(X) \}
    \]
    is a finite set.
\end{thm}
\begin{proof}
See \cite[Theorem~3.1]{FS}.
\end{proof}

\begin{thmC}\label{Zdense}
Let $f: \af^n \rightarrow \af^n$ be a regular polynomial automorphism defined over a number field $K$. Then, the set of periodic point is Zariski dense.
\end{thmC}
\begin{proof}
Let $K$ be a finitely generated field over $\qq$ containing all coefficients of $f$. It is enough to show that $\Per(f)$ is Zariski dense on $\pp^n_{\overline{K}}$.

Lemma~\ref{pick} says that $\{ g_1 = f^{l_2}, g_2 = f^{-l_1}\}$ is strongly regular of the same degree. Also, $\{ g_1^m, g_2^{m}\}$ is strongly regular with the same degree for all natural number $m$ by Corollary~\ref{again}. Define a sequence of morphisms;
\[
\phi_m (X) : = (g_1^m(X), g_2^{m}(X)).
\]

Let $\mathbb{B}$ be a finite set of places such that $\{g_1, g_2\}$ have good reduction at $v\not \in \mathbb{B}$. Then $\phi_1$ also has good reduction. Let $\mathbf{k}=\mathbf{k}_v$ be the corresponding residue field. Consider a morphism
\[
\widetilde{\iota} : \pp^n_{\overline{\mathbf{k}}} \rightarrow \pp^{2n}_{\overline{\mathbf{k}}}, \quad [X_0, X_1, \cdots, X_n ] \mapsto [X_0, X_1, \cdots, X_n , X_1 \cdots, X_n].
\]
    Then, by Theorem~\ref{finite}, we have finite sets
    \[
    \widetilde{A}_m = \widetilde{A}_{m,v} := \{X \in \pp^n_{\overline{\mathbf{k}}} ~|~ \widetilde{\phi_m} (X) = \widetilde{\iota}(X) \}
    \]
    in $\pp^n_{\overline{\mathbf{k}}}$.

    Let $l = |l_1 - l_2|$ and pick a point
    \[
    X \in \Per_{lm}(\widetilde{f})_v = \{ X' \in \af^n_{\overline{\mathbf{k}}}~|~ \widetilde{f}^{lm}(X') = X'\}.
    \]
    Then, we have
    \[
    \widetilde{g_1}^m(X) = \widetilde{f}^{ml_1}(X) = \widetilde{f}^{-ml_2}(X) = \widetilde{g_2}^m(X).
    \]
    So,
    \[
    \Per_{lm}\bigr( \widetilde{f} \bigr)_v \subset  \widetilde{A}_m .
    \]
    Moreover,
    \[
    \widetilde{f} : \af^n_{\mathbf{k}'} \rightarrow \af^n_{\mathbf{k}'}
    \]
    is still an automorphism for any finite extension $\mathbf{k}'$ of $\mathbf{k}$. Thus, all points are periodic points: $ \Per(\widetilde{f})_v = \af^n_{\overline{\mathbf{k}}}$.

    Therefore, the union of $\widetilde{A}_m$ contains a Zariski dense subset of $\pp^n_{\overline{\mathbf{k}}}$;
    \[
    \af^n_{\overline{\mathbf{k}}} \subset \bigcup \widetilde{A}_{m}.
    \]

Now, let $\mathcal{W}_1, \mathcal{W}_2$ be arithmetic models of $\pp^n_{K}$ and $\pp^{2n}_{K}$, let ${\phi}^A_m, {\iota}^A$ be extensions of $\phi_m$ and $\iota$ and let
\[
\mathcal{A}_m = \{\mathcal{X} \in \mathcal{W}_1 ~|~ {\phi}^A_m(\mathcal{X})= {\iota}^A(\mathcal{X}) \}.
\]
It is a nontrivial subvariety of $\mathcal{W}_1$. And, it cannot be of dimension 0 because it contains points in all but finitely many special fibers of $\mathcal{W}$, where $f, f^{-1}$ have good reduction. Thus, it is of dimension at least 1. But, each fiber only contains finitely many intersection with $\mathcal{A}_m$, which is $\widetilde{A}_{m,v}$. Therefore, $\bigcup \mathcal{A}_m$ consists of infinitely many horizontal sections whose intersection with all special fibers over $\mathbb{B}^c$ are Zariski dense. Hence, $\bigcup \mathcal{A}_m$ consists of infinitely many horizontal sections whose intersection with all fiber, especially with the generic fiber, are Zariski dense.
\end{proof}

    \begin{cor}\label{gs}
        Let $f :\af^n \rightarrow \af^n$ be a regular polynomial automorphism. Then, there is a generic and small  sequence $\{ x_m\} \subset \Per(f)$.
    \end{cor}
    \begin{proof}
        Because of Theorem~C, we can find a generic sequence $\{x_m\} \subset \Per(f)$. So, we only need to show that it is small:
        let $\overline{L} = (\ox_{\pp^n}, ||\cdot||_S)$ be the metrized line bundle with the dynamical adelic metric defined in Section~6 and let $h_{\overline{L}}$ be the corresponding arithmetic height. Because ${h}_{\overline{\mathcal{L}}}(x_m)=0$ if $x_m \in \Per(f)$, we only have to show that
        ${h}_{\overline{L}}(\pp^n)=0$ to claim that $\{ x_m\}$ is a small sequence.

        \cite[Theorem 1.10]{Z1} says that
        \[
        e_1(\overline{L})   \geq
        {h}_{\overline{L}}(\pp^n) \geq \dfrac{1}{n} \sum_{i=1}^n e_i (\overline{L})
        \]
        where
        \[
        e_i(\overline{L})
        = \sup_{\substack{W \subsetneq \pp^n \\ \operatorname{Codim} W = i}} \inf_{X \in \pp^n \setminus W} \widehat{h}_{\overline{L}} (X).
        \]
        By Theorem~\ref{Zdense}, $\Per(f)$ is Zariski Dense in $\pp^n$ and hence $e_i(\overline{L}) =0$ for all $i = 1, \cdots, n$ and hence ${h}_{\overline{L}}(\pp^n) =0$.
    \end{proof}

    Now, we can prove the desired theorem:

    \begin{thmA}
        Let $f:\af^n \rightarrow \af^n$ be a regular polynomial automorphism defined over a number field $K$ and let $v\in M_K$. Then, there exists an $f$-invariant probability measure $\mu_{f,v}$ on $\operatorname{Berk}\bigl( \pp^n_{\mathbb{C}_v} \bigr)$ such that the set of periodic points of $f$ is equidistributed with respect to $\mu_{f,v}$: let $\{ x_m\} $ be a sequence of periodic points of $f$ such that $|\{ x_m\} \cap W|< \infty$ for any proper subvariety $W$ of $\pp^n$. Then
        \[
        \dfrac{1}{\ord x_m} \sum_{y \in \Gamma_{x_m}} \delta_y ~\text{weakly converges to}~\mu_{f,v}
        \]
       where $\Gamma_{x_m}$ is the Galois orbit of $x_m$ and $\delta_y$ is a Dirac measure at $y$.
    \end{thmA}
    \begin{proof}
        By Corollary~\ref{gs}, we can find a generic and small sequence consisting of periodic points of $f$. Thus, by Theorem~B, we have the equidistribution of periodic points of $f$ with respect to $\mu_{f,v} = \mu_{\{f^{l_1}, f^{-l_2}\},v}$.
    \end{proof}

    \begin{cor}
    $\mu_{f,v}$ in Theorem~A is $f$-invariant measure.
    \end{cor}
    \begin{proof}
    If $\{x_m\}$ is a generic and small sequence, then so is $\{ f(x_m)\}$. Also, $f\Gamma_{x_m} = \Gamma_{f(x_m)}$. Therefore,
    \[
    \int \phi f^*\mu_{f,v} = \int \phi \circ f  \mu_{f,v} = \lim \sum_{y\in \Gamma_{x_m}} \int \phi \circ f  \delta_{y}
    = \lim \sum_{f(y)\in \Gamma_{f(x_m)}} \int \phi \delta_{f(y)} = \int \phi \mu_{f,v}.
    \]
    \end{proof}

    If $n=2$ and $v \infty$, then $\mu_{f,\infty} = \mu_+ \wedge \mu_-$ by \cite{BLS}. If $f$ is a shift-like automorphism on $\cc^n$ and $l_1, l_2$ are integers defined in Proposition~7.1, then
    $\mu_{f,v} = \mu_+^{l_2} \wedge \mu_-^{l_1}$ by \cite{BP}. Recently, Dinh and Sibony showed it is true for any regular polynomial automorphism:

    \begin{thm}[Dinh-Sibony]
    Let $v\in M_K$ be an archimedean place. Then $\mu_{f,v}$ in Theorem~A is
    \[
    \mu_{f,v} = \mu_+^{l_2} \wedge \mu_-^{l_1}
    \]
    where $\mu_+$, $\mu_-$ are the Green currents defined by $f, f^{-1}$ and $l_1,l_2$ are integers defined in Proposition~7.1.
    \end{thm}
    \begin{proof}
    We can find the complete proof in \cite{DS2}. Or, we check the following fact in the next section briefly:
    \[
    \mu_{f,v} = \bigl( dd^c G_S \bigr)^n.
    \]
    Since $f$ is regular, we have
    \[
    \bigl( dd^c G_S \bigr)^n = \mu_+^{l_2} \wedge \mu_-^{l_1}
    \]
    by \cite[Proposition~2.9]{GS}.
   \end{proof}

\end{document}